\documentclass[final]{my_siamltex}

\usepackage{amsmath,amsbsy,amssymb,latexsym}
\usepackage{amssymb}

\usepackage{graphicx}
\usepackage{url}

\usepackage{cite}

\newtheorem{remark}[theorem]{Remark}
\newtheorem{example}{Example}

\newcommand{\R}{\mathbb{R}}

\newcommand{\N}{\mathbb{N}}
\newcommand{\opab}{{^CD}^{\alpha,\beta}_{\gamma}}

\newcommand{\opabp}{^CD^{{\alpha}_{1},{\beta}_{1}}_{{\gamma}_{1}}}
\newcommand{\opabi}{^CD^{\alpha_{i-1},\beta_{i-1}}_{\gamma_{i-1}}}
\newcommand{\opabN}{^CD^{\alpha_N,\beta_N}_{\gamma_N}}

\newcommand{\opbac}{D^{\beta,\alpha}_{1-\gamma}}
\newcommand{\opbaci}{D^{\beta_{i-1},\alpha_{i-1}}_{1-\gamma_{i-1}}}

\newcommand{\caprb}{{^C_xD_b^\beta}}
\newcommand{\capla}{{^C_aD_x^\alpha}}

\newcommand{\intlb}{{_aI_x^{1-\beta}}}
\newcommand{\intlbi}{{_aI_x^{1-\beta_{i-1}}}}
\newcommand{\intra}{{_xI_b^{1-\alpha}}}
\newcommand{\intrai}{{_xI_b^{1-\alpha_{i-1}}}}

\begin{document}

\myref

\title{Multiobjective fractional variational calculus\\
in terms of a combined Caputo derivative\thanks{This work
was partially supported by the \emph{Portuguese Foundation for
Science and Technology} through the R\&D unit
\emph{Center for Research and Development in Mathematics and Applications}.}}

\author{Agnieszka B. Malinowska\thanks{Faculty of Computer Science, 
Bia{\l}ystok University of Technology, 15-351 Bia\l ystok, Poland 
({\tt a.malinowska@pb.edu.pl}). 
Partially supported by BUT Grant S/WI/2/2011.} 
\and Delfim F. M. Torres\thanks{Department of
Mathematics, University of Aveiro, 3810-193 Aveiro, Portugal 
({\tt delfim@ua.pt}). Partially supported by FCT (Portugal) 
through the project UTAustin/MAT/0057/2008.}}

\maketitle


\begin{abstract}
The study of fractional variational problems in terms
of a combined fractional Caputo derivative is introduced.
Necessary optimality conditions of Euler--Lagrange type
for the basic, isoperimetric, and Lagrange variational problems
are proved, as well as transversality and sufficient optimality conditions.
This allows to obtain necessary and sufficient Pareto optimality conditions
for multiobjective fractional variational problems.
\end{abstract}






\pagestyle{myheadings}
\thispagestyle{plain}
\markboth{A. B. MALINOWSKA AND D. F. M. TORRES}{MULTIOBJECTIVE FRACTIONAL VARIATIONAL CALCULUS}


\section{Introduction}

There is an increasing interest in the study of dynamic systems of fractional
(where ``fractional'' actually means ``non-integer'') order.
Extending derivatives and integrals from integer to non-integer order
has a firm and longstanding theoretical foundation.
Leibniz mentioned this concept in a letter to L'Hopital
over three hundred years ago. Following L'Hopital's and Leibniz's
first inquisition, fractional calculus was primarily a study reserved to
the best minds in mathematics. Fourier, Euler, and Laplace are among the
many that contributed to the development of fractional calculus.
Along the history, many found, using their own notation and methodology,
definitions that fit the concept of a non-integer order integral or
derivative. The most famous of these definitions among mathematicians,
that have been popularized in the literature of fractional calculus, are the ones of
Riemann--Liouville and Grunwald--Letnikov. On the other hand,
the most intriguing and useful applications of fractional
derivatives and integrals in engineering and science have been found
in the past one hundred years.
In some cases, the mathematical notions evolved in order
to better meet the requirements of physical reality.
The best example of this is the Caputo fractional derivative,
nowadays the most popular fractional operator among engineers and applied scientists,
obtained by reformulating the ``classical'' definition of
Riemann--Liouville fractional derivative in order to be possible
to solve fractional initial value problems
with standard initial conditions \cite{Dot:Porto}.
Particularly in the last decade of the XX century,
numerous applications and physical manifestations
of fractional calculus have been found.
Fractional differentiation is nowadays recognized as a good tool
in various different fields: physics, signal processing, fluid mechanics,
viscoelasticity, mathematical biology, electrochemistry, chemistry,
economics, engineering, and control theory (see, \textrm{e.g.},
\cite{debnath,Diethelm,ferreira,hilfer,hilfer2,kulish,magin,metzler,Oustaloup,TM,Zas}).

The fractional calculus of variations was born
in 1996 with the work of Riewe \cite{rie,rie97},
and is nowadays a subject under strong current research (see
\cite{Almeida2,MyID:154,MyID:152,MyID:179,El-Nabulsi:Torres07,%
R:T:08,Frederico:Torres07,Frederico:Torres08,G:D:10,MyID:181} and
references therein). The fractional calculus of variations extends
the classical variational calculus by considering fractional
derivatives into the variational integrals to be extremized. This
occurs naturally in many problems of physics and mechanics, in order
to provide more accurate models of physical phenomena (see,
\textrm{e.g.}, \cite{R:A:D:10,Atanackovic}). The aims of this paper
are twofold. Firstly, we extend the notion of Caputo fractional
derivative to the fractional derivative $\opab$, which is a convex
combination of the left Caputo fractional derivative of order
$\alpha$ and the right Caputo fractional derivative of order
$\beta$. This idea goes back at least as far as \cite{klimek}, where
based on the Riemann--Liouville fractional derivatives, the
symmetric fractional derivative was introduced. Klimek's approach
\cite{klimek} is obtained in our framework as a particular case, by
choosing parameter $\gamma$ to be $1/2$. Although the symmetric
fractional derivative of Riemann--Liouville introduced by Klimek is
a useful tool in the description of some nonconservative models,
this type of differentiation does not seems suitable for all kinds
of variational problems. Indeed, the hypothesis that admissible
trajectories $y$ have continuous symmetric fractional derivatives
implies that $y(a)=y(b)=0$ (\textrm{cf.} \cite{Ross}). Therefore,
the advantage of the fractional Caputo-type derivative $\opab$ here
introduced lies in the fact that using this derivative we can
describe a more general class of variational problems. It is also
worth pointing out that the fractional derivative $\opab$ allows to
generalize the results presented in \cite{AlmeidaTorres}. Our second
aim is to introduce the subject of multiobjective fractional
variational problems. This seems to be a completely open area of
research, never considered before in the literature. Knowing the
importance and relevance of multiobjective problems of the calculus
of variations in physics, engineering, and economics (see
\cite{Engwerda,BasiaDelfim,MyID:093,BorisBookII,Sophos,Wang} and the
references given there), and the usefulness of fractional
variational problems, we trust that the results now obtained will
open interesting possibilities for future research. Main results of
the the paper provide methods for identifying Pareto optimal
solutions. Necessary and sufficient Pareto optimality conditions are
obtained by converting a multiobjective fractional variational
problem into a single or a family of single fractional variational
problems with an auxiliary scalar functional, possibly depending on
some parameters.

The paper is organized as follows. Section~\ref{sec2} presents some
preliminaries on fractional calculus, essentially to fix notations.
In Section~\ref{main:1} we introduce the fractional
derivative $\opab$ and provide the necessary concepts and results
needed in the sequel. Our main results are stated and proved in
Section~\ref{ssec:pro} and Section~\ref{main:2}. The fractional
variational problems under our consideration are formulated in terms
of the fractional derivative $\opab$. We discuss the fundamental
concepts of a variational calculus such as the the Euler--Lagrange
equations for the elementary (Subsection~\ref{ssec:EL}),
isoperimetric (Subsection~\ref{sec:iso}), and Lagrange
(Subsection~\ref{sec:lagr}) problems, as well as sufficient optimality
(Subsection~\ref{ssec:suf}) and transversality
(Subsection~\ref{sec:tran}) conditions. Section~\ref{main:2} deals with
the multiobjective fractional variational calculus. We present Pareto
optimality conditions (Subsection~\ref{sec:par:op}) and examples
illustrating our results (Subsection~\ref{sec:ex}).


\section{Fractional calculus}
\label{sec2}

In this section we review the necessary definitions and facts from
fractional calculus. For more on the subject we refer the reader to
the books \cite{kilbas,Oldham,Podlubny,samko}.
Let $f\in L_1([a,b])$  and $0<\alpha<1$. We begin by defining the left and
the right Riemann--Liouville Fractional Integrals (RLFI) of order
$\alpha$ of a function $f$. The left RLFI is given by
\begin{equation}
\label{RLFI1}
{_aI_x^\alpha}f(x):=\frac{1}{\Gamma(\alpha)}\int_a^x
(x-t)^{\alpha-1}f(t)dt,\quad x\in[a,b],
\end{equation}
and the right RLFI by
\begin{equation}
\label{RLFI2}
{_xI_b^\alpha}f(x):=\frac{1}{\Gamma(\alpha)}\int_x^b(t-x)^{\alpha-1}
f(t)dt,\quad x\in[a,b],
\end{equation}
where $\Gamma(\cdot)$ represents the Gamma function, \textrm{i.e.},
$$
\Gamma(z):=\int_0^\infty t^{z-1} \mathrm{e}^{-t}\, dt.
$$
Moreover, ${_aI_x^0}f={_xI_b^0}f=f$ if $f$ is a continuous function.
The left and the right Riemann--Liouville derivatives are defined
with the help of the respective fractional integrals.
The left Riemann--Liouville Fractional Derivative (RLFD) is given by
\begin{equation}
\label{RLFD1}
{_aD_x^\alpha}f(x):=\frac{1}{\Gamma(1-\alpha)}\frac{d}{dx}\int_a^x
(x-t)^{-\alpha}f(t)dt=\frac{d}{dx}{_aI_x^{1-\alpha}}f(x),\quad
x\in[a,b],
\end{equation}
and the right RLFD by
\begin{equation}
\label{RLFD2}
{_xD_b^\alpha}f(x):=\frac{-1}{\Gamma(1-\alpha)}\frac{d}{dx}\int_x^b
(t-x)^{-\alpha}
f(t)dt=\left(-\frac{d}{dx}\right){_xI_b^{1-\alpha}}f(x),\quad
x\in[a,b].
\end{equation}
Let $f\in AC([a,b])$, where $AC([a,b])$ represents the space of
absolutely continuous functions on $[a,b]$. Then the Caputo fractional
derivatives are defined as follows:
the left Caputo Fractional Derivative (CFD) by
\begin{equation}
\label{CFD1}
{^C_aD_x^\alpha}f(x):=\frac{1}{\Gamma(1-\alpha)}\int_a^x
(x-t)^{-\alpha}\frac{d}{dt}f(t)dt={_aI_x^{1-\alpha}}\frac{d}{dx}f(x),
\quad x\in[a,b],
\end{equation}
and the right CFD by
\begin{equation}
\label{CFD2}
{^C_xD_b^\alpha}f(x):=\frac{-1}{\Gamma(1-\alpha)}\int_x^b
(t-x)^{-\alpha}
\frac{d}{dt}f(t)dt={_xI_b^{1-\alpha}}\left(-\frac{d}{dx}\right)f(x),
\quad x\in[a,b],
\end{equation}
where $\alpha$ is the order of the derivative.
The operators \eqref{RLFI1}--\eqref{CFD2} are obviously linear. We
now present the rule of fractional integration by parts for RLFI
(see, \textrm{e.g.}, \cite{int:partsRef}). Let $0<\alpha<1$, $p\geq1$,
$q \geq 1$, and $1/p+1/q\leq1+\alpha$. If $g\in L_p([a,b])$ and
$f\in L_q([a,b])$, then
\begin{equation}
\label{ipi} \int_{a}^{b} g(x){_aI_x^\alpha}f(x)dx =\int_a^b f(x){_x
I_b^\alpha} g(x)dx.
\end{equation}
In the discussion to follow, we will also need the following
formulae for fractional integrations by parts:
\begin{equation}
\label{ip}
\begin{split}
\int_{a}^{b}  g(x) \, {^C_aD_x^\alpha}f(x)dx &=\left.f(x){_x
I_b^{1-\alpha}} g(x)\right|^{x=b}_{x=a}+\int_a^b f(x){_x D_b^\alpha}
g(x)dx,\\
\int_{a}^{b}  g(x) \, {^C_xD_b^\alpha}f(x)dx &=\left.-f(x){_a
I_x^{1-\alpha}} g(x)\right|^{x=b}_{x=a}+\int_a^b f(x){_a D_x^\alpha}
g(x)dx.
\end{split}
\end{equation}
They can be derived using equations \eqref{RLFD1}--\eqref{CFD2},
the identity \eqref{ipi}, and performing integration by parts.


\section{The fractional operator $\mathbf{\opab}$}
\label{main:1}

Let  $\alpha, \beta \in(0,1)$ and $\gamma\in [0,1]$ . We define the
fractional derivative operator $\opab$ by
\begin{equation}
\label{op}
\opab :=\gamma \, \capla + (1-\gamma) \, \caprb \, ,
\end{equation}
which acts on $f\in AC([a,b])$ in the expected way:
\begin{equation*}
\opab f(x)=\gamma \capla f(x) + (1-\gamma) \caprb f(x).
\end{equation*}
Note that ${^CD^{\alpha,\beta}_{0}} f(x)=\caprb f(x)$
and ${^CD^{\alpha,\beta}_{1}} f(x)=\capla f(x)$.
The operator \eqref{op} is obviously linear. Using equations
\eqref{ip} it is easy to derive the following rule of fractional
integration by parts for $\opab$:
\begin{multline}
\label{byparts}
\int_{a}^{b}  g(x) \, \opab f(x)dx
=\gamma\left[f(x)\intra g(x)\right]^{x=b}_{x=a}\\
+(1-\gamma)\left[-f(x)\intlb g(x)\right]^{x=b}_{x=a}
+\int_a^b f(x)\opbac g(x)dx,
\end{multline}
where $\opbac:=(1-\gamma){_aD_x^\beta}+\gamma {_xD_b^\alpha}$.
Let $N\in\N$ and $\mathbf{f}=[f_1,\ldots,f_N]:[a,b]\rightarrow
\R^{N}$ with $f_i\in AC([a,b])$, $i=1,\ldots,N$; $\alpha, \beta,
\gamma\in \mathbb{R}^N$ with $\alpha_i, \beta_i \in(0,1)$ and
$\gamma_i\in [0,1]$, $i=1,\ldots,N$. Then,
$$
\opab \mathbf{f}(x):= \left[\opabp f_1(x),\ldots,\opabN f_N(x)\right].
$$
Let $\mathbf{D}$ denote the set of all functions
$\mathbf{y}:[a,b]\rightarrow \R^{N}$ such that $\opab \mathbf{y}$
exists and is continuous on the interval $[a,b]$. We endow
$\mathbf{D}$ with the following norm:
\begin{equation*}
    \|\mathbf{y}\|_{1,\infty}:=\max_{a\leq x \leq
    b}\|\mathbf{y}(x)\|+\max_{a\leq x \leq
    b}\|\opab\mathbf{y}(x)\|,
\end{equation*}
where $\|\cdot\|$ is a norm in $\R^N$.
Along the work we denote by $\partial_i K$, $i=1,\ldots,M$ ($M\in
\N$), the partial derivative of function $K:\R^M\rightarrow \R$ with
respect to its $i$th argument.
Let $\mathbf{\lambda} \in \R^r$. For simplicity of notation we
introduce the operators $[\mathbf{y}]^{\alpha,\beta}_{\gamma}$
and $_{\lambda}\{\mathbf{y}\}^{\alpha,\beta}_{\gamma}$ by
\begin{equation*}
\begin{split}
[\mathbf{y}]^{\alpha,\beta}_{\gamma}(x) &:= \left(x,\mathbf{y}(x),\opab
\mathbf{y}(x)\right) \, , \\
_{\lambda}\{\mathbf{y}\}^{\alpha,\beta}_\gamma(x) &:= \left(x,\mathbf{y}(x),\opab
\mathbf{y}(x),\lambda_1,\ldots,\lambda_r\right) \, .
\end{split}
\end{equation*}


\section{Calculus of variations via $\mathbf{\opab}$}
\label{ssec:pro}

We are concerned with the problem of finding the minimum of a functional
$\mathcal{J}: \mathcal{D}\rightarrow \R$, where $ \mathcal{D}$ is a
subset of $\mathbf{D}$. The formulation of a problem of the calculus
of variations requires two steps: the specification of a performance
criterion, and the statement of physical constraints
that should be satisfied. The performance criterion $\mathcal{J}$,
also called cost functional (or objective), must be specified
for evaluating quantitatively the performance of the system under study.
We consider the following cost:
\begin{equation*}
\mathcal{J}(\mathbf{y})=\int_a^b L[\mathbf{y}]^{\alpha,\beta}_{\gamma}(x) \, dx,
\end{equation*}
where $x\in [a,b]$ is the independent variable;
$\mathbf{y}(x)\in \R^N$ is a real vector variable, the functions
$\mathbf{y}$ are generally called trajectories or curves; $\opab
\mathbf{y}(x)\in \R^N$ stands for the fractional derivative of
$\mathbf{y}(x)$; and $L\in C^1([a,b]\times\mathbb{R}^{2N};
\mathbb{R})$ is the Lagrangian.

Enforcing constraints in the optimization problem reduces the set of
candidate functions and leads to the following definition.

\begin{definition}
A trajectory $\mathbf{y}\in \mathbf{D}$ is said to be an admissible
trajectory, provided it satisfies all the constraints of the problem along
the interval $[a,b]$. The set of admissible trajectories is defined as
$\mathcal{D}:=\{\mathbf{y}\in \mathbf{D}:\mathbf{y} \mbox{ is admissible}\}$.
\end{definition}

We now define what is meant by a minimum of $\mathcal{J}$ on
$\mathcal{D}$.

\begin{definition}
A trajectory $\bar{\mathbf{y}}\in \mathcal{D}$ is said to be a local
minimizer for $\mathcal{J}$ on $\mathcal{D}$, if there exists
$\delta>0$ such that $\mathcal{J}(\bar{\mathbf{y}})\leq
\mathcal{J}(\mathbf{y})$ for all $\mathbf{y}\in \mathcal{D}$ with
$\|\mathbf{y}-\bar{\mathbf{y}}\|_{1,\infty}<\delta$.
\end{definition}

The concept of variation of a functional is central to the solution
of problems of the calculus of variations.

\begin{definition}
The first variation of $\mathcal{J}$ at $\mathbf{y}\in \mathbf{D}$
in the direction $\mathbf{h}\in \mathbf{D}$ is defined as
\begin{equation*}
\delta\mathcal{J}(\mathbf{y};\mathbf{h})
:=\lim_{\varepsilon\rightarrow 0}\frac{\mathcal{J}(\mathbf{y}
+\varepsilon\mathbf{h})-\mathcal{J}(\mathbf{y})}{\varepsilon}
=\left.\frac{\partial}{\partial\varepsilon}\mathcal{J}(\mathbf{y}
+\varepsilon\mathbf{h})\right|_{\varepsilon=0},
\end{equation*}
provided the limit exists.
\end{definition}

\begin{definition}
A direction $\mathbf{h}\in \mathbf{D}$, $\mathbf{h}\neq 0$, is said
to be an admissible variation for $\mathcal{J}$ at $\mathbf{y}\in \mathcal{D}$ if
\begin{itemize}
\item[(i)] $\delta\mathcal{J}(\mathbf{y};\mathbf{h})$ exists; and
\item[(ii)] $\mathbf{y}+\varepsilon\mathbf{h}\in \mathcal{D}$ for
all sufficiently small $\varepsilon$.
\end{itemize}
\end{definition}

The following well known result offers a necessary optimality
condition for the problems of the calculus of variations,
based on the concept of variations.

\begin{theorem}[see, \textrm{e.g.}, Proposition~5.5 of \cite{Trout}]
\label{nesse_con}
Let $\mathcal{J}$ be a functional defined on $\mathcal{D}$.
Suppose that $\mathbf{y}$ is a local minimizer for
$\mathcal{J}$ on $\mathcal{D}$. Then,
$\delta\mathcal{J}(\mathbf{y};\mathbf{h})=0$ for each admissible
variation $\mathbf{h}$ at $\mathbf{y}$.
\end{theorem}


\subsection{Elementary problem of the $\mathbf{\opab}$ fractional calculus of variations}
\label{ssec:EL}

Let us begin with the following fundamental problem:
\begin{equation}
\label{Funct1} \mathcal{J}(\mathbf{y})
=\int_a^b L[\mathbf{y}]^{\alpha,\beta}_{\gamma}(x) \, dx \longrightarrow \min\\
\end{equation}
over all $\mathbf{y}\in \mathbf{D}$ satisfying the boundary
conditions
\begin{equation}
\label{boun2}
\mathbf{y}(a)=\mathbf{y}^{a}, \quad
\mathbf{y}(b)=\mathbf{y}^{b},
\end{equation}
where $\mathbf{y}^{a},\mathbf{y}^{b}\in \R^N$ are given.
The next theorem gives the fractional Euler--Lagrange equation for the
problem \eqref{Funct1}--\eqref{boun2}.

\begin{theorem}
\label{Theo E-L1}
Let $\mathbf{y}=(y_1,\ldots,y_N)$ be a local
minimizer to problem \eqref{Funct1}--\eqref{boun2}. Then,
$\mathbf{y}$ satisfies the following system of $N$ fractional Euler--Lagrange
equations:
\begin{equation}
\label{E-L1}
\partial_i L[\mathbf{y}]^{\alpha,\beta}_{\gamma}(x)
+\opbaci \partial_{N+i} L[\mathbf{y}]^{\alpha,\beta}_{\gamma}(x)=0,
\quad i=2,\ldots N+1,
\end{equation}
for all $x\in[a,b]$.
\end{theorem}

\begin{proof}
Suppose that $\mathbf{y}$ is a local minimizer for $\mathcal{J}$.
Let $\mathbf{h}$ be an arbitrary admissible variation for problem
\eqref{Funct1}--\eqref{boun2}, \textrm{i.e.}, $h_i(a)=h_i(b)=0$,
$i=1,\ldots,N$. Based on the differentiability properties of $L$ and
Theorem~\ref{nesse_con}, a necessary condition for $\mathbf{y}$ to
be a local minimizer is given by
$$
\left.\frac{\partial}{\partial\varepsilon}\mathcal{J}(\mathbf{y}
+\varepsilon\mathbf{h})\right|_{\varepsilon=0} = 0 \, ,
$$
that is,
\begin{equation}
\label{eq:FT}
\int_a^b\Biggl[\sum_{i=2}^{N+1}\partial_i L[\mathbf{y}]^{\alpha,\beta}_{\gamma}(x) h_{i-1}(x)
+\sum_{i=2}^{N+1}\partial_{N+i}L[\mathbf{y}]^{\alpha,\beta}_{\gamma}(x)\opabi h_{i-1}(x)\Biggr]dx=0.
\end{equation}
Using formulae \eqref{byparts} of integration by parts in the
second term of the integrand function, we get
\begin{multline}
\label{eq:aft:IP}
\int_a^b\left[\sum_{i=2}^{N+1}\partial_i
L[\mathbf{y}]^{\alpha,\beta}_{\gamma}(x)+\opbaci\partial_{N+i}
L[\mathbf{y}]^{\alpha,\beta}_{\gamma}(x)\right]h_{i-1}(x)dx\\
+\gamma\left.\left[\sum_{i=2}^{N+1}h_{i-1}(x)\left(\intrai
\partial_{N+i}L[\mathbf{y}]^{\alpha,\beta}_{\gamma}(x)\right)\right]\right|^{x=b}_{x=a}\\
-(1-\gamma)\left.\left[\sum_{i=2}^{N+1}h_{i-1}(x)\left(\intlbi
\partial_{N+i}L[\mathbf{y}]^{\alpha,\beta}_{\gamma}(x)\right)\right]\right|^{x=b}_{x=a}=0.
\end{multline}
Since $h_i(a)=h_i(b)=0$, $i=1,\ldots,N$, by the fundamental lemma of
the calculus of variations we deduce that
\begin{equation*}
\partial_i L[\mathbf{y}]^{\alpha,\beta}_{\gamma}(x)
+\opbaci\partial_{N+i}L[\mathbf{y}]^{\alpha,\beta}_{\gamma}(x)=0,
\quad i=2,\ldots,N+1,
\end{equation*}
for all $x\in[a,b]$.
\end{proof}

Observe that if $\alpha$ and $\beta$ go to $1$, then $\capla$ can be
replaced with $\frac{d}{dx}$ and $\caprb$ with $-\frac{d}{dx}$ (see
\cite{Podlubny}). Thus, if $\gamma=1$ or $\gamma=0$, then for
$\alpha,\beta \rightarrow 1$ we obtain a corresponding result in the
classical context of the calculus of variations (see, \textrm{e.g.},
\cite[Proposition~6.1]{Trout}).


\subsection{Fractional transversality conditions}
\label{sec:tran}

Let $l\in \{1,\ldots,N\}$. Assume that $\mathbf{y}(a)=\mathbf{y}^a$,
$y_i(b)=y_i^b$, $i=1,\ldots,N$ , $i\neq l$, but $y_l(b)$ is free.
Then, $h_l(b)$ is free and by equations \eqref{E-L1} and
\eqref{eq:aft:IP} we obtain
\begin{equation*}
\Bigl[\gamma {_xI_b^{1-\alpha_l}}
\partial_{N+l+1}L[\mathbf{y}]^{\alpha,\beta}_{\gamma}(x)
\left.-(1-\gamma) {_aI_x^{1-\beta_l}}\partial_{N+1+l}
L[\mathbf{y}]^{\alpha,\beta}_{\gamma}(x)\Bigr]\right|_{x=b}=0,
\end{equation*}
where $\alpha, \beta, \gamma \in \mathbb{R}$. Let us consider now
the case when $\mathbf{y}(a)=\mathbf{y}^a$, $y_i(b)=y_i^b$,
$i=1,\ldots,N$ , $i\neq l$, and  $y_l(b)$ is free but restricted by
a terminal condition $y_l(b)\leq y^{b}_l$. Then, in the optimal
solution $\mathbf{y}$, we have two possible types of outcome:
$y_l(b)< y^{b}_l$ or $y_l(b)= y^{b}_l$. If $y_l(b)< y^{b}_l$, then
there are admissible neighboring paths with terminal value both
above and below $y_l(b)$, so that $h_l(b)$ can take either sign.
Therefore, the transversality conditions is
\begin{equation}
\label{tran:1}
\Bigl[\gamma {_xI_b^{1-\alpha_l}}
\partial_{N+l+1}L[\mathbf{y}]^{\alpha,\beta}_{\gamma}(x)
\left.-(1-\gamma) {_aI_x^{1-\beta_l}}\partial_{N+1+l}
L[\mathbf{y}]^{\alpha,\beta}_{\gamma}(x)\Bigr]\right|_{x=b}=0
\end{equation}
for $y_l(b)< y_l^{b}$. The other outcome $y_l(b)= y^{b}_l$ only
admits the neighboring paths with terminal value $\tilde{y}_l(b)\leq
y_l(b)$. Assuming, without loss of generality, that $h_l(b)\geq 0$,
this means that $\varepsilon \leq 0$. Hence, the transversality
condition, which has it root in the first order condition
\eqref{eq:FT}, must be changed to an inequality. For a minimization
problem, the $\leq$ type of inequality is called for, and we obtain
\begin{equation}
\label{tran:2}
\Bigl[\gamma {_xI_b^{1-\alpha_l}}
\partial_{N+l+1}L[\mathbf{y}]^{\alpha,\beta}_{\gamma}(x)
\left.-(1-\gamma) {_aI_x^{1-\beta_l}}\partial_{N+1+l}
L[\mathbf{y}]^{\alpha,\beta}_{\gamma}(x)\Bigr]\right|_{x=b} \leq 0
\end{equation}
for $y_l(b)= y^{b}_l$. Combining \eqref{tran:1} and \eqref{tran:2},
we may write the following transversality condition for a
minimization problem:
\begin{multline*} 
\Bigl[\gamma {_xI_b^{1-\alpha_l}}
\partial_{N+l+1}L[\mathbf{y}]^{\alpha,\beta}_{\gamma}(x)
\left.-(1-\gamma) {_aI_x^{1-\beta_l}}\partial_{N+1+l}
L[\mathbf{y}]^{\alpha,\beta}_{\gamma}(x)\Bigr]\right|_{x=b}
\leq 0, \quad y_l(b)\leq y^{b}_l,\\
(y_l(b)-y^{b}_l)
\Bigl[\gamma {_xI_b^{1-\alpha_l}}
\partial_{N+l+1}L[\mathbf{y}]^{\alpha,\beta}_{\gamma}(x)
\left.-(1-\gamma) {_aI_x^{1-\beta_l}}\partial_{N+1+l}
L[\mathbf{y}]^{\alpha,\beta}_{\gamma}(x)\Bigr]\right|_{x=b}=0.
\end{multline*}


\subsection{The $\mathbf{\opab}$ fractional isoperimetric problem}
\label{sec:iso}

Let us consider now the isoperimetric problem that consists of
minimizing \eqref{Funct1} over all $\mathbf{y}\in \mathbf{D}$
satisfying $r$ isoperimetric constraints
\begin{equation}
\label{cons2:iso} \mathcal{G}^j(\mathbf{y})=\int_{a}^{b}
G^j[\mathbf{y}]^{\alpha,\beta}_{\gamma}(x)dx=l_j, \quad j=1,\ldots,r,
\end{equation}
where $G^j\in C^1([a,b]\times\mathbb{R}^{2N}; \mathbb{R})$,
$j=1,\ldots,r$, and boundary conditions \eqref{boun2}.

Necessary optimality conditions for isoperimetric problems can be
obtained by the following general theorem.

\begin{theorem}[see, \textrm{e.g.}, Theorem~2 of \cite{G:H} on p.~91]
\label{nes:iso}
Let $\mathcal{J},\mathcal{G}^{1},\ldots,\mathcal{G}^r$ be functionals
defined in a neighborhood of $\mathbf{y}$ and having continuous
first variations in this neighborhood. Suppose that $\mathbf{y}$ is
a local minimizer of \eqref{Funct1} subject to the boundary conditions
\eqref{boun2} and the isoperimetric constrains \eqref{cons2:iso}.
Assume that there are functions $\mathbf{h}^{1},
\ldots,\mathbf{h}^{r}\in \mathbf{D}$ such that the matrix
$A=(a_{kl})$, $a_{kl}:=\delta\mathcal{G}^{k}(\mathbf{y};\mathbf{h}^l)$,
has maximal rank $r$. Then there exist constants
$\lambda_{1},\ldots,\lambda_{r}\in \R$ such that the functional
\begin{equation*}
\mathcal{F}:=\mathcal{J}-\sum_{j=1}^{r}\lambda _{j}\mathcal{G}^{j}
\end{equation*}
satisfies
\begin{equation}\label{var}
\delta\mathcal{F}(\mathbf{y};\mathbf{h})=0
\end{equation}
for all $\mathbf{h}\in \mathbf{D}$
\end{theorem}

Suppose now that assumptions of Theorem~\ref{nes:iso} hold. Then,
equation \eqref{var} is fulfilled for every $\mathbf{h} \in
\mathbf{D}$. Let us consider function $\mathbf{h}$ such that
$\mathbf{h}(a)=\mathbf{h}(b)=0$. Then, we have
\begin{multline*}
0=\delta \mathcal{F}(\mathbf{y};\mathbf{h})
=\frac{\partial}{\partial \varepsilon}\mathcal{F}(\mathbf{y}
+\varepsilon\mathbf{h})|_{\varepsilon=0}\\
=\int_a^b\Biggl[\sum_{i=2}^{N+1}\partial_i
F _{\lambda}\{\mathbf{y}\}^{\alpha,\beta}_\gamma(x)h_{i-1}(x)
+\sum_{i=2}^{N+1}\partial_{N+i} F _{\lambda}\{\mathbf{y}\}^{\alpha,\beta}_\gamma(x)
\opabi h_{i-1}(x)\Biggr]dx,
\end{multline*}
where the function $F:[a,b]\times \R^{2N}\times \R^r \rightarrow \R$
is defined by
$$
F _{\lambda}\{\mathbf{y}\}^{\alpha,\beta}_\gamma(x):=L[\mathbf{y}]^{\alpha,\beta}_{\gamma}(x)
-\sum_{j=1}^{r}\lambda_{j} G^{j}[\mathbf{y}]^{\alpha,\beta}_{\gamma}(x).
$$
On account of the above, and similarly in spirit to the proof of
Theorem~\ref{Theo E-L1}, we obtain
\begin{equation}
\label{ele}
\partial_{i} {_{\lambda}}F\{\mathbf{y}\}^{\alpha,\beta}_\gamma(x)
+\opbaci \partial_{N+i} {_{\lambda}}F\{\mathbf{y}\}^{\alpha,\beta}_\gamma(x)=0,
\quad i=2,\ldots N+1.
\end{equation}

Therefore, we have the following necessary optimality condition for
the fractional isoperimetric problems:

\begin{theorem}
\label{Th:B:EL-CV} Let assumptions of Theorem~\ref{nes:iso} hold. If
$\mathbf{y}$ is a local minimizer to the isoperimetric problem given
by \eqref{Funct1},\eqref{boun2} and \eqref{cons2:iso}, then
$\mathbf{y}$ satisfies the system of $N$ fractional Euler--Lagrange
equations \eqref{ele} for all $x\in[a,b]$.
\end{theorem}

Suppose now that constraints \eqref{cons2:iso} are characterized by
inequalities
\begin{equation*}
\mathcal{G}^j(\mathbf{y})
=\int_{a}^{b} G^j[\mathbf{y}]^{\alpha,\beta}_{\gamma}(x)dx
\leq l_j, \quad j=1,\ldots,r.
\end{equation*}
In this case we can set
\begin{equation*}
\int_{a}^{b}\left(G^j[\mathbf{y}]^{\alpha,\beta}_{\gamma}(x)
-\frac{l_j}{b-a}\right)dx
+\int_{a}^{b}(\phi_j(x))^2dx=0,
\end{equation*}
$j=1,\ldots,r$, where $\phi_j$ have the some continuity properties
as $y_i$. Therefore, we obtain the following problem:
\begin{equation*}
\hat{\mathcal{J}}(y)=\int_a^b \hat{L}(x,\mathbf{y}(x),
\opab\mathbf{y}(x),\phi_{1}(x),\ldots,\phi_{r}(x)) 
\, dx \longrightarrow \min
\end{equation*}
subject to $r$ isoperimetric constraints
\begin{equation*}
\int_{a}^{b}\left[G^j[\mathbf{y}]^{\alpha,\beta}_{\gamma}(x)
-\frac{l_j}{b-a}+(\phi_j(x))^2\right]dx=0, \quad j=1,\ldots,r,
\end{equation*}
and boundary conditions \eqref{boun2}. Assuming that assumptions of
Theorem~\ref{Th:B:EL-CV} are satisfied, we conclude that there exist
constants $\lambda_{j}\in \R$, $j=1,\ldots,r$, for which the system
of equations
\begin{multline}
\label{iso:1:L:EL}
\opbaci \partial_{N+i}\hat{F}(x,\mathbf{y}(x),
\opab\mathbf{y}(x),\lambda_1,
\ldots,\lambda_r,\phi_{1}(x),\ldots,\phi_{r}(x))\\
+\partial_i\hat{F}(x,\mathbf{y}(x),
\opab\mathbf{y}(x),\lambda_1,\ldots,\lambda_r,\phi_{1}(x),\ldots,\phi_{r}(x))=0,
\end{multline}
$i=2,\ldots, N+1$,
$\hat{F}=\hat{L}+\sum_{j=1}^r\lambda_j(G^j-\frac{l_j}{b-a}+\phi_j^2)$ and
\begin{equation}
\label{iso:2:L:EL}
\lambda_j\phi_j(x)=0,  \quad j=1,\ldots,r,
\end{equation}
hold for all $x\in[a,b]$. Note that it is enough to assume that the
regularity condition holds for the constraints which are active at
the local minimizer $\mathbf{y}$ (constraint $\mathcal{G}^k$ is active 
at $\mathbf{y}$ if $\mathcal{G}^k(\mathbf{y})=l_k$). Indeed,
suppose that $l<r$ constrains, say
$\mathcal{G}^1,\ldots,\mathcal{G}^l$ for simplicity, are active at
the local minimizer $\mathbf{y}$, and there are functions
$\mathbf{h}^{1},\ldots,\mathbf{h}^{l}\in \mathbf{D}$ such that the
matrix
\begin{equation*}
B=(b_{kj}),\quad
b_{kj}:=\delta\mathcal{G}^{k}(\mathbf{y};\mathbf{h}^j),\quad
k,j=1,\ldots,l<r
\end{equation*}
has maximal rank $l$. Since the inequality constraints
$\mathcal{G}^{l+1},\ldots,\mathcal{G}^r$ are inactive, the condition
\eqref{iso:2:L:EL} is trivially satisfied by taking
$\lambda_{l+1}=\cdots=\lambda_{r}=0$. On the other hand, since the
inequality constraints $\mathcal{G}^1,\ldots,\mathcal{G}^l$ are
active and satisfy a regularity condition at $\mathbf{y}$, the
conclusion that there exist constants $\lambda_{j}\in \R$,
$j=1,\ldots,r$, such that \eqref{iso:1:L:EL} holds follow from
Theorem~\ref{Th:B:EL-CV}. Moreover, \eqref{iso:2:L:EL} is trivially
satisfied for $j=1,\ldots,l$.


\subsection{The $\mathbf{\opab}$ fractional Lagrange problem}
\label{sec:lagr}

Let us consider the following Lagrange problem, which consists of
minimizing \eqref{Funct1} over all $\mathbf{y}\in \mathbf{D}$
satisfying $r$ independent constraints ($r<N$)
\begin{equation}\label{cons2}
 G^j[\mathbf{y}]^{\alpha,\beta}_{\gamma}(x)=0, \quad j=1,\ldots,r,
\end{equation}
and boundary conditions \eqref{boun2}. In mechanics, constraints of
type \eqref{cons2} are called nonholonomic. By the independence of
the $r$ constraints $G^j\in C^1([a,b]\times\mathbb{R}^{2N};
\mathbb{R})$ it is meant that it should exist a nonvanishing
Jacobian determinant of order $r$, such as
$\left|\frac{\partial(G^1,\ldots,G^r)}{\partial(p_{N+2},\ldots,p_{N+2+r})}\right|\neq
0$. Of course, any $r$ of $p_j$, $j=N+2,...,2N+1$, can be used, not
necessarily the first $r$.

\begin{theorem}
\label{lagrange} A function $\mathbf{y}$ which is a solution to
problem \eqref{Funct1},\eqref{boun2} subject to $r$ independent
constraints ($r<N$) \eqref{cons2} satisfies, for suitably chosen
functions $\lambda_j$, $j=1,\ldots,r$, the system of $N$ fractional
Euler--Lagrange equations
\begin{equation*}\label{L:EL}
\partial_i F[\mathbf{y},\mathbf{\lambda}]^{\alpha,\beta}_{\gamma}(x)
+\opbaci \partial_{N+i}F[\mathbf{y},\mathbf{\lambda}]^{\alpha,\beta}_{\gamma}(x)=0,
\quad x\in[a,b], \quad i=2,\ldots, N+1,
\end{equation*}
where
$F[\mathbf{y},\mathbf{\lambda}]^{\alpha,\beta}_{\gamma}(x)
=L[\mathbf{y}]^{\alpha,\beta}_{\gamma}(x)
+\sum_{j=1}^r\lambda_j(x)G^j[\mathbf{y}]^{\alpha,\beta}_{\gamma}(x)$.
\end{theorem}

\begin{proof}
Suppose that $\mathbf{y}=(y_1,\ldots, y_N)$ is the solution to
problem defined by \eqref{Funct1},\eqref{boun2}, and \eqref{cons2}.
Let $\mathbf{h}=(h_1,\ldots,h_N)$ be an arbitrary admissible
variation, \textrm{i.e.}, $h_i(a)=h_i(b)=0$, $i=1,\ldots,N$, and
$G^j[\mathbf{y}+\varepsilon \mathbf{h}]^{\alpha,\beta}_{\gamma}(x)=0$,
$j=1,\ldots,r,$ where $\varepsilon\in\mathbb{R}$ is a small parameter.
Because $\mathbf{y}=(y_1,\ldots, y_N)$ is a solution to problem defined by
\eqref{Funct1},\eqref{boun2}, and \eqref{cons2}, it follows that
$$
\left.\frac{\partial}{\partial\varepsilon}\mathcal{J}(\mathbf{y}
+\varepsilon\mathbf{h})\right|_{\varepsilon=0} = 0 \, ,
$$
that is,
\begin{equation}
\label{eq:FTL}
\int_a^b\Biggl[\sum_{i=2}^{N+1}\partial_i L[\mathbf{y}]^{\alpha,\beta}_{\gamma}(x)
h_{i-1}(x) +\sum_{i=2}^{N+1}\partial_{N+i}L[\mathbf{y}]^{\alpha,\beta}_{\gamma}(x)
\opabi h_{i-1}(x)\Biggr]dx=0,
\end{equation}
and, for $j=1,\ldots,r$,
\begin{equation}
\label{nes:2}
\sum_{i=2}^{N+1}\partial_i G^j[\mathbf{y}]^{\alpha,\beta}_{\gamma}(x)h_{i-1}(x)
+\sum_{i=2}^{N+1}\partial_{N+i}G^j[\mathbf{y}]^{\alpha,\beta}_{\gamma}(x)
\opabi h_{i-1}(x)=0.
\end{equation}
Multiplying the $j$th equation of the system \eqref{nes:2} by the
unspecified function $\lambda_j(x)$, for all $j=1,\ldots,r$,
integrating with respect to $x$, and adding the left-hand sides
(all equal to zero for any choice of the $\lambda_j$) to the
integrand of \eqref{eq:FTL}, we obtain
\begin{equation*}
\begin{split}
&\int_a^b \left[\sum_{i=2}^{N+1}\left(\partial_iL[\mathbf{y}]^{\alpha,\beta}_{\gamma}(x)
+\sum_{j=1}^r\lambda_j(x)\partial_iG^j[\mathbf{y}]^{\alpha,\beta}_{\gamma}(x)\right)h_{i-1}(x)\right.\\
& \left. \quad +\sum_{i=2}^{N+1}\left(\partial_{N+i}L[\mathbf{y}]^{\alpha,\beta}_{\gamma}(x)
+\sum_{j=1}^r\lambda_j(x)\partial_{N+i}G^j[\mathbf{y}]^{\alpha,\beta}_{\gamma}(x)\right)(\opabi
h_{i-1}(x))\right]dx\\
&=\int_a^b\left[\sum_{i=2}^{N+1}\partial_i F[\mathbf{y},\mathbf{\lambda}]^{\alpha,\beta}_{\gamma}(x)
h_{i-1}(x) +\sum_{i=2}^{N+1}\partial_{N+i}F[\mathbf{y},\mathbf{\lambda}]^{\alpha,\beta}_{\gamma}(x)
(\opabi h_{i-1}(x))\right]dx\\
&=0,
\end{split}
\end{equation*}
where
$F[\mathbf{y},\mathbf{\lambda}]^{\alpha,\beta}_{\gamma}(x)
=L[\mathbf{y}]^{\alpha,\beta}_{\gamma}(x) +\sum_{j=1}^r\lambda(x)
G^j[\mathbf{y}]^{\alpha,\beta}_{\gamma}(x)$.
Integrating by parts,
\begin{equation}
\label{nes:3}
\int_a^b\left[\sum_{i=2}^{N+1}\partial_i F[\mathbf{y},\mathbf{\lambda}]^{\alpha,\beta}_{\gamma}(x)
+\opbaci \partial_{N+i} F[\mathbf{y},\mathbf{\lambda}]^{\alpha,\beta}_{\gamma}(x)\right] h_{i-1}(x)dx=0.
\end{equation}
Because of \eqref{nes:2}, we cannot regard the $N$ functions
$h_1,\ldots,h_N$ as being free for arbitrary choice. There is a
subset of $r$ of these functions whose assignment is restricted by
the assignment of the remaining $(N-r)$. We can assume, without loss
of generality, that $h_1,\ldots,h_r$ are the functions of the set
whose dependence upon the choice of the arbitrary
$h_{r+1},\ldots,h_N$ is governed by \eqref{nes:2}. We now assign the
functions $\lambda_1,\ldots,\lambda_r$ to be the set of $r$
functions that make vanish (for all $x$ between $a$ and $b$) the
coefficients of $h_1,\ldots,h_r$ in the integrand of \eqref{nes:3}.
That is, $\lambda_1,\ldots,\lambda_r$ are chosen so as to satisfy
\begin{equation}
\label{nes:4}
\partial_i F[\mathbf{y},\mathbf{\lambda}]^{\alpha,\beta}_{\gamma}(x)
+\opbaci\partial_{N+i} F[\mathbf{y},\mathbf{\lambda}]^{\alpha,\beta}_{\gamma}(x)=0,
\quad i=2,\ldots,r+1, \quad x\in[a,b].
\end{equation}
With this choice \eqref{nes:3} gives
\begin{equation*}
\int_a^b\left[\sum_{i=r+2}^{N+1}
\partial_i F[\mathbf{y},\mathbf{\lambda}]^{\alpha,\beta}_{\gamma}(x)
+\opbaci\partial_{N+i}
F[\mathbf{y},\mathbf{\lambda}]^{\alpha,\beta}_{\gamma}(x)\right] h_{i-1}(x)dx=0.
\end{equation*}
Since the functions $h_{r+1},\ldots,h_N$  are arbitrary, we may
employ the fundamental lemma of the calculus of variations to conclude that
\begin{equation}
\label{nes:6}
\partial_i F[\mathbf{y},\mathbf{\lambda}]^{\alpha,\beta}_{\gamma}(x)
+\opbaci\partial_{N+i}F[\mathbf{y},\mathbf{\lambda}]^{\alpha,\beta}_{\gamma}(x)=0,
\end{equation}
$i=r+1,\ldots,N+1$, for all $x\in[a,b]$.
\end{proof}

\begin{remark}
In order to determine the $(N+r)$ unknown functions
$y_1,\ldots,y_n$, $\lambda_1,\ldots,\lambda_r$, we must consider
the system of $(N+r)$ equations, consisting of \eqref{cons2},
\eqref{nes:4}, and \eqref{nes:6}, together with the $2N$ boundary
conditions \eqref{boun2}.
\end{remark}

Assume now that the constraints, instead of \eqref{cons2},
are characterized by inequalities:
\begin{equation*}
\label{cons:ineq}
G^j[\mathbf{y}]^{\alpha,\beta}_{\gamma}(x)\leq 0, \quad j=1,\ldots,r.
\end{equation*}
In this case we can set
\begin{equation*}
\label{cons:ineq:chan}
 G^j[\mathbf{y}]^{\alpha,\beta}_{\gamma}(x)+(\phi_j(x))^2=0, \quad j=1,\ldots,r,
\end{equation*}
where $\phi_j$ have the some continuity properties as $y_i$.
Therefore, we obtain the following problem:
\begin{equation}
\label{Funct2:ineq}
\hat{\mathcal{J}}(y)=\int_a^b \hat{L}(x,\mathbf{y}(x),
\opab\mathbf{y}(x), \phi_1(x),\ldots,\phi_r(x)) \, dx 
\longrightarrow \text{min}
\end{equation}
subject to $r$ independent constraints ($r<N$)
\begin{equation}
\label{cons2:ineq}
 G^j[\mathbf{y}]^{\alpha,\beta}_{\gamma}(x)+(\phi_j(x))^2=0, \quad j=1,\ldots,r,
\end{equation}
and boundary conditions \eqref{boun2}. Applying
Theorem~\ref{lagrange} we get the following result.

\begin{theorem}
\label{lagrange:ineq}
A set of functions $y_1,\ldots, y_N$, $\phi_1,\ldots,\phi_r$, which
is a solution to problem \eqref{Funct2:ineq}--\eqref{cons2:ineq},
satisfies, for suitably chosen $\lambda_j$, $j=1,\ldots,r$, the
following system of equations:
\begin{multline*}
\opbaci \partial_{N+i}\hat{F}(x,\mathbf{y}(x),\opab\mathbf{y}(x),\lambda_1(x),
\ldots,\lambda_r(x),\phi_1(x),\ldots,\phi_r(x))\\
+ \partial_i\hat{F}(x,\mathbf{y}(x),\opab\mathbf{y}(x),\lambda_1(x),
\ldots,\lambda_r(x),\phi_1(x),\ldots,\phi_r(x)) =0,
\end{multline*}
$i=2,\ldots, N+1$, where $\hat{F}=\hat{L}+\sum_{j=1}^r\lambda_j(G^j+\phi_j^2)$, and
$\lambda_j(x)\phi_j(x)=0$,  $j=1,\ldots,r$, hold for all $x\in[a,b]$.
\end{theorem}


\subsection{Sufficient condition of optimality}
\label{ssec:suf}

In this section we provide sufficient optimality conditions for the
elementary and the isoperimetric problem of the $\opab$ fractional
calculus of variations. Similarly to what happens in the classical
calculus of variations, some conditions of convexity are in order.

\begin{definition}
Given a function $f\in C^1([a,b]\times\mathbb{R}^{2N}; \mathbb{R})$,
we say that $f(\underline x,\mathbf{y},\mathbf{v})$ is jointly
convex in $(\mathbf{y},\mathbf{v})$, if
\begin{equation*}
f(x,\mathbf{y}+\mathbf{y}^0,\mathbf{v}+\mathbf{v}^0)-f(x,\mathbf{y},\mathbf{v})
\geq \sum_{i=2}^{N+1}\partial_i
f(x,\mathbf{y},\mathbf{v})y_{i-1}^0
+\sum_{i=2}^{N+1}\partial_{N+i}f(x,\mathbf{y},\mathbf{v})v_{i-1}^0
\end{equation*}
for all
$(x,\mathbf{y},\mathbf{v})$,$(x,\mathbf{y}+\mathbf{y}^0,\mathbf{v}
+\mathbf{v}^0)\in [a,b]\times\mathbb{R}^{2N}$.
\end{definition}

\begin{theorem}
\label{suff}
Let $L(\underline x,\mathbf{y},\mathbf{v})$ be jointly convex in
$(\mathbf{y},\mathbf{v})$. If $\mathbf{y}$ satisfies the system of
$N$ fractional Euler--Lagrange equations \eqref{E-L1}, then
$\mathbf{y}$ is a global minimizer to problem
\eqref{Funct1}--\eqref{boun2}.
\end{theorem}

\begin{proof}
The proof is similar to the proof of Theorem~3.3 in \cite{MalTor}.
\end{proof}

\begin{theorem}
\label{suff:iso}
Let $F(\underline
x,\mathbf{y},\mathbf{v},\bar{\mathbf{\lambda}})=L(\underline
x,\mathbf{y},\mathbf{v})-\sum_{j=1}^{r}\bar{\lambda}_{j}G^{j}(\underline
x,\mathbf{y},\mathbf{v})$ be jointly convex in
$(\mathbf{y},\mathbf{v})$, for some constants
$\bar{\lambda}_{j}\in\mathbb{R}$, $j=1,\ldots,r$. If $\mathbf{y}^0$
satisfies the system of $N$ fractional Euler--Lagrange equations
\eqref{ele}, then $\mathbf{y}^0$ is a minimizer to the isoperimetric problem
defined by \eqref{Funct1},\eqref{boun2} and \eqref{cons2:iso}.
\end{theorem}

\begin{proof}
By Theorem~\ref{suff}, $\mathbf{y}^0$ minimizes
$\int_a^b F _{\bar{\lambda}}\{\mathbf{y}\}^{\alpha,\beta}_{\gamma}(x) \, dx$. That is, for all
functions satisfying condition \eqref{boun2} we have
\begin{multline*}
\int_a^b L[\mathbf{y}]^{\alpha,\beta}_{\gamma}(x)\,
dx-\sum_{j=1}^{r}\bar{\lambda}_{j}\int_a^bG^{j}[\mathbf{y}]^{\alpha,\beta}_{\gamma}(x)\, dx\\
\geq \int_a^b L[\mathbf{y}^0]^{\alpha,\beta}_{\gamma}(x)\,
dx-\sum_{j=1}^{r}\bar{\lambda}_{j}\int_a^bG^{j}[\mathbf{y}^0]^{\alpha,\beta}_{\gamma}(x)\, dx.
\end{multline*}
Restricting to the isoperimetric constraints \eqref{cons2:iso}, we obtain that
\begin{equation*}
\int_a^b L[\mathbf{y}]^{\alpha,\beta}_{\gamma}(x)\, dx
-\sum_{j=1}^{r}\bar{\lambda}_{j}l_{j}\,
dx\geq \int_a^b L[\mathbf{y}^0]^{\alpha,\beta}_{\gamma}(x)\, dx
-\sum_{j=1}^{r}\bar{\lambda}_{j}l_{j}\, dx.
\end{equation*}
Therefore,
\begin{equation*}
\int_a^b L[\mathbf{y}]^{\alpha,\beta}_{\gamma}(x)\, dx
\geq \int_a^b L[\mathbf{y}^0]^{\alpha,\beta}_{\gamma}(x)\, dx
\end{equation*}
as desired.
\end{proof}

Choosing $r=1$ in Theorem~\ref{suff:iso} one can easily obtain
\cite[Theorem~3.10]{AlmeidaTorres}.


\section{Multiobjective fractional optimization}
\label{main:2}

Multiobjective optimization is a natural extension of the
traditional optimization of a single-objective function. If the
objective functions are commensurate, minimizing one-objective
function minimizes all criteria and the problem can be solved using
tradicional optimization techniques. However, if the objective
functions are incommensurate, or competing, then the minimization of
one objective function requires a compromise in another objective.
Here we consider multiobjective fractional
variational problems with a finite number
$d\geq 1$ of objective (cost) functionals
\begin{equation}
\label{Funct:mul}
\left(\mathcal{J}^1(\mathbf{y}),\ldots,\mathcal{J}^d(\mathbf{y})\right)
=\left(\int_a^b L^1[\mathbf{y}]^{\alpha,\beta}_{\gamma}(x) \, dx ,
\ldots, \int_a^b L^d[\mathbf{y}]^{\alpha,\beta}_{\gamma}(x) \, dx \right)\longrightarrow \min
\end{equation}
subject to the boundary conditions
\begin{equation}
\label{boun1:mul}
\mathbf{y}(a)=\mathbf{y}^{a},
\quad \mathbf{y}(b)=\mathbf{y}^{b},
\end{equation}
$\mathbf{y}^{a},\mathbf{y}^{b}\in \R^N$,
and $r$ ($r<N$) independent constraints
\begin{equation}
\label{cons:mul}
G^j[\mathbf{y}]^{\alpha,\beta}_{\gamma}(x)\leq 0, \quad j=1,\ldots,r,
\end{equation}
where $L^i,G^j\in C^1([a,b]\times\mathbb{R}^{2N}; \mathbb{R})$,
$i=1,\ldots,d$, $j=1,\ldots,r$. We would like to find a function
$\mathbf{y}\in \mathbf{D}$, satisfying constraints \eqref{boun1:mul}
and \eqref{cons:mul}, that renders the minimum value to each
functional $\mathcal{J}^i$, $i=1,\ldots, d$, simultaneously. The
competition between objectives gives rise to the necessity
of distinguish between the difference of multiobjective optimization
and traditional single-objective optimization.
Competition causes the lack of complete order
for multiobjective optimization problems. The concept of
Pareto optimality is therefore used to characterize a solution
to the multiobjective optimization problem. For the usefulness
of variational analysis and Pareto optimal allocations
in welfare economics, we refer the reader to \cite{Boris2005}.
We define
\begin{equation*}
\mathcal{E}:=\{\mathbf{y}\in \mathbf{D}:\mathbf{y} \mbox{ satisyies
conditions \eqref{boun1:mul} and \eqref{cons:mul}}\}.
\end{equation*}

\begin{definition}
\label{def:paret}
A function $\bar{\mathbf{y}}\in \mathcal{E}$ is called a Pareto
optimal solution to problem \eqref{Funct:mul}--\eqref{cons:mul}
if does not exist $\mathbf{y}\in \mathcal{E}$ with
\begin{equation*}
\forall
i\in\{1,\ldots,d\}:\mathcal{J}^i(\mathbf{y})\leq\mathcal{J}^i(\bar{\mathbf{y}})
\quad \wedge \quad \exists
i\in\{1,\ldots,d\}:\mathcal{J}^i(\mathbf{y})<\mathcal{J}^i(\bar{\mathbf{y}}).
\end{equation*}
\end{definition}

Definition~\ref{def:paret} introduces the notion of
\emph{global Pareto optimality}.
Another important concept is the one of \emph{local Pareto optimality}.

\begin{definition}
\label{def:locparet}
A function $\bar{\mathbf{y}}\in \mathcal{E}$ is called a local
Pareto optimal solution to problem
\eqref{Funct:mul}--\eqref{cons:mul} if there exists $\delta >0$
for which does not exist $\mathbf{y}\in \mathcal{E}$ with
$\|\mathbf{y}-\bar{\mathbf{y}}\|_{1,\infty}<\delta$ and
\begin{equation*}
\forall
i\in\{1,\ldots,d\}:\mathcal{J}^i(\mathbf{y})\leq\mathcal{J}^i(\bar{\mathbf{y}})\quad
\wedge \quad \exists
i\in\{1,\ldots,d\}:\mathcal{J}^i(\mathbf{y})<\mathcal{J}^i(\bar{\mathbf{y}}).
\end{equation*}
\end{definition}

Naturally, any global Pareto optimal solution is locally Pareto optimal.
For enhanced notions of Pareto optimality
of constrained multiobjective problems,
the reader is referred to \cite{Boris2010}.


\subsection{Fractional Pareto optimality conditions}
\label{sec:par:op}

We obtain a sufficient condition for Pareto optimality by modifying the
original multiobjective fractional problem
\eqref{Funct:mul}--\eqref{cons:mul} into the following weighting
problem:
\begin{equation}
\label{wei:prob} \sum_{i=1}^{d}w_{i}\int_a^b L^i[\mathbf{y}]^{\alpha,\beta}_{\gamma}(x) \,
dx \longrightarrow \min
\end{equation}
subject to $\mathbf{y}\in\mathcal{E}$, where $w_{i}\geq 0$ for all
$i=1,\ldots,d$, and $\sum_{i=1}^{d}w_{i}=1$.

\begin{theorem}
\label{pareto:suf}
The solution of the weighting problem \eqref{wei:prob} is Pareto
optimal if the weighting coefficients are positive, that is,
$w_{i}>0$ for all $i=1,\ldots,d$. Moreover, the unique solution of
the weighting problem \eqref{wei:prob} is Pareto optimal.
\end{theorem}

\begin{proof}
Let $\bar{\mathbf{y}}\in \mathcal{E}$ be a solution to problem
\eqref{wei:prob} with $w_{i}>0$ for all $i=1,\ldots,d$. Suppose that
$\bar{\mathbf{y}}$ is not Pareto optimal. Then, there exists
$\mathbf{y}$ such that
$\mathcal{J}^i(\mathbf{y})\leq\mathcal{J}^i(\bar{\mathbf{y}})$ for
all $i=1,\ldots,d$ and
$\mathcal{J}^j(\mathbf{y})<\mathcal{J}^j(\bar{\mathbf{y}})$ for at
least one $j$. Since $w_{i}>0$ for all $i=1,\ldots,d$, we have
$\sum_{i=1}^{d}w_i\mathcal{J}^i(\mathbf{y})<\sum_{i=1}^{d}w_i\mathcal{J}^i(\bar{\mathbf{y}})$.
This contradicts the minimality of $\bar{\mathbf{y}}$. Now, let
$\bar{\mathbf{y}}$ be the unique solution to \eqref{wei:prob}. If
$\bar{\mathbf{y}}$ is not Pareto optimal, then
$\sum_{i=1}^{d}w_i\mathcal{J}^i(\mathbf{y})\leq\sum_{i=1}^{d}w_i\mathcal{J}^i(\bar{\mathbf{y}})$.
This contradicts the uniqueness of $\bar{\mathbf{y}}$.
\end{proof}

Therefore, by varying the weights over the unit simplex
$\{w=(w_1,\ldots,w_d): w_i\geq 0, \sum_{i=1}^{d}w_{i}=1\}$ ones
obtains, in principle, different Pareto optimal solutions. The next
theorem provides a necessary and sufficient condition for Pareto
optimality. The result is analogous to that valid
for the finite dimensional case (see, \textrm{e.g.},
Chapter~3.1 and Chapter~3.3 of \cite{Miet}).

\begin{theorem}
\label{pareto:nes:suf}
A function $\bar{\mathbf{y}}\in \mathcal{E}$ is Pareto optimal
to problem \eqref{Funct:mul}--\eqref{cons:mul} if and only
if it is a solution to the scalar fractional variational problem
\begin{equation*}
\int_a^b L^i[\mathbf{y}]^{\alpha,\beta}_{\gamma}(x) \, dx \longrightarrow \min
\end{equation*}
subject to $\mathbf{y}\in \mathcal{E}$ and
\begin{equation*}
\int_a^b L^j[\mathbf{y}]^{\alpha,\beta}_{\gamma}(x) \, dx
\leq \int_a^b L^j[\bar{\mathbf{y}}]^{\alpha,\beta}_{\gamma}(x) \, dx,
\quad j=1,\ldots,d,\quad j\neq i,
\end{equation*}
for each $i=1,\ldots,d$.
\end{theorem}

\begin{proof}
Suppose that $\bar{\mathbf{y}}$ is Pareto optimal. Then
$\bar{\mathbf{y}}\in \mathcal{C}_k=\{\mathbf{y}\in \mathcal{E}:
\mathcal{J}^j(\mathbf{y})\leq\mathcal{J}^j(\bar{\mathbf{y}}),j=1,\ldots,d,j\neq
k\}$ for all $k$, so $\mathcal{C}_k\neq \emptyset$. If
$\bar{\mathbf{y}}$ does not minimize $\mathcal{J}^k(\mathbf{y})$ on
the constrained set $\mathcal{C}_k$ for some $k$, then there exists
$\mathbf{y}\in \mathcal{E}$ such that
$\mathcal{J}^k(\mathbf{y})<\mathcal{J}^k(\bar{\mathbf{y}})$ and
$\mathcal{J}^j(\mathbf{y})\leq\mathcal{J}^j(\bar{\mathbf{y}})$ for
all $j\neq k$. This contradicts the Pareto optimality of
$\bar{\mathbf{y}}$. Now, suppose that $\bar{\mathbf{y}}$ minimize
each $\mathcal{J}^k(\mathbf{y})$ on the constrained set
$\mathcal{C}_k$. If $\bar{\mathbf{y}}$ is not Pareto optimal, then
there exists $\mathbf{y}$ such that
$\mathcal{J}^i(\mathbf{y})\leq\mathcal{J}^i(\bar{\mathbf{y}})$ for
all $i=1,\ldots,d$ and
$\mathcal{J}^j(\mathbf{y})<\mathcal{J}^j(\bar{\mathbf{y}})$ for at
least one $j$. This contradicts the minimality of $\mathbf{y}$ for
$\mathcal{J}^j(\mathbf{y})$ on $\mathcal{C}_j$.
\end{proof}

\begin{remark}
\label{pareto:nes}
For a function $\mathbf{y}\in \mathcal{E}$ to be Pareto optimal
to problem \eqref{Funct:mul}--\eqref{cons:mul}, it is
necessary to be a solution to the fractional isoperimetric problems
\begin{equation*}
\int_a^b L^i[\mathbf{y}]^{\alpha,\beta}_{\gamma}(x) \, dx \longrightarrow \min
\end{equation*}
subject to $\mathbf{y}\in \mathcal{E}$ and
\begin{equation*}
\int_a^b L^j[\mathbf{y}]^{\alpha,\beta}_{\gamma}(x) \, dx
= \int_a^b L^j[\bar{\mathbf{y}}]^{\alpha,\beta}_{\gamma}(x) \, dx,
\quad j=1,\ldots,d,\quad j\neq i,
\end{equation*}
for all $i=1,\ldots,d$. Therefore, necessary optimality conditions
for the fractional isoperimetric problems (see Theorem~\ref{Th:B:EL-CV})
are also necessary for fractional Pareto optimality.
\end{remark}


\subsection{Examples}
\label{sec:ex}

We illustrate our results with two multiobjective
fractional variational problems.

\begin{example}
Let $\bar{y}(x)=E_{\alpha}(x^{\alpha})$, $x\in [0,1]$, where
$E_{\alpha}$ is the Mittag--Leffler function:
$$
E_{\alpha}(z)=\sum_{k=1}^{\infty}\frac{z^k}{\Gamma(\alpha k+1)},
\quad z\in\mathbb{R}, \quad \alpha >0.
$$
When $\alpha =1$, the Mittag--Leffler function
is simply the exponential function: $E_1(x)=\mathrm{e}^x$.
We note that the left Caputo fractional derivative of
$\bar{y}$ is $\bar{y}$ (\textrm{cf.} \cite{kilbas}, p.~98):
$$
{^C_0D_x^\alpha}\bar{y}(x)=\bar{y}(x).
$$
Consider the following multiobjective fractional variational problem
($N=1$, $\gamma=1$, and $d=2$):
\begin{equation}
\label{ex:1:pr}
\left(\mathcal{J}^1(y),\mathcal{J}^2(y)\right)
=\left(\int_0^1 ({^C_0D_x^\alpha}y(x))^2 \, dx ,
\int_0^1 \bar{y}(x){^C_0D_x^\alpha}y(x)  \, dx \right)\longrightarrow \min
\end{equation}
subject to
\begin{equation}
\label{ex:1:pr:bc}
y(0)=0,\quad y(1)=E_{\alpha}(1).
\end{equation}
Observe that $\bar{y}$ satisfies the necessary Pareto optimality
conditions (see Remark~\ref{pareto:nes}). Indeed, as shown in
\cite[Example~1]{AlmeidaTorres}, $\bar{y}$ is a solution to the
isoperimetric problem
\begin{equation*}
\mathcal{J}^1(y)=\int_0^1 ({^C_0D_x^\alpha}y(x))^2 \, dx \longrightarrow \min
\end{equation*}
subject to
\begin{equation*}
\int_0^1 \bar{y}(x){^C_0D_x^\alpha}y(x) \, dx =\int_0^1
(\bar{y}(x))^2 \, dx.
\end{equation*}
Consider now the following fractional isoperimetric problem:
\begin{equation*}
\mathcal{J}^2(y)=\int_0^1 \bar{y}(x){^C_0D_x^\alpha}y(x)\, dx \longrightarrow \min
\end{equation*}
subject to
\begin{equation*}
\int_0^1 ({^C_0D_x^\alpha}y(x))^2  \, dx =\int_0^1 ({^C_0D_x^\alpha}
\bar{y}(x))^2 \, dx.
\end{equation*}
Let us apply Theorem~\ref{nes:iso}. The equality ${_xD_1^\alpha}y(x)=0$
holds if and only if $y(x)=d(1-x)^{\alpha-1}$ with $d\in \R$ (see
\cite[Corollary~2.1]{kilbas}). Hence, $\bar{y}$ does not satisfy
equation ${_xD_1^\alpha}({^C_0D_x^\alpha}y)=0$.
The augmented function is
\begin{equation*}
F _{\lambda}\{\mathbf{y}\}^{\alpha,\beta}_\gamma(x)
=\bar{y}(x){^C_0D_x^\alpha}y(x)
-\lambda ({^C_0D_x^\alpha}y(x))^2,
\end{equation*}
and the corresponding fractional Euler--Lagrange equation gives
\begin{equation*}
{_xD_1^\alpha}(\bar{y}(x)-2\lambda{^C_0D_x^\alpha}y(x))=0.
\end{equation*}
A solution to this equation is $\lambda=\frac{1}{2}$ and
$y=\bar{y}$. Therefore, by Remark~\ref{pareto:nes}, $y=\bar{y}$ is a
candidate Pareto optimal solution to problem
\eqref{ex:1:pr}--\eqref{ex:1:pr:bc}.
\end{example}

\begin{example}
\label{ex:2}
Consider the following multiobjective fractional variational problem:
\begin{equation}
\label{ex:2:pr} \left(\mathcal{J}^1(y),\mathcal{J}^2(y)\right)
=\left(\int_0^1 \frac{1}{2}({^C_0D_x^\alpha}y(x)-f(x))^2 \, dx ,
\int_0^1 \frac{1}{2}({^C_0D_x^\alpha}y(x))^2  \, dx \right)\longrightarrow \min
\end{equation}
subject to
\begin{equation}
\label{ex:2:pr:bc}
y(0)=0,\quad y(1)=\chi,\quad \chi \in \R,
\end{equation}
where $f$ is a fixed function. In this case we have
$N=1$, $\gamma=1$, and $d=2$. By Theorem~\ref{pareto:suf}, Pareto
optimal solutions to problem \eqref{ex:2:pr}--\eqref{ex:2:pr:bc}
can be found by considering the family of problems
\begin{equation}
\label{ex:2:fp}
w\int_0^1 \frac{1}{2}({^C_0D_x^\alpha}y(x)-f(x))^2 \, dx
+ (1-w)\int_0^1 \frac{1}{2}{^C_0D_x^\alpha}y(x)  \, dx\longrightarrow \min
\end{equation}
subject to
\begin{equation}
\label{ex:2:fp:bc}
y(0)=0,\quad y(1)=\chi,\quad \chi \in \R,
\end{equation}
where $w\in[0,1]$. Let us now fix $w$. By Theorem~\ref{Theo E-L1},
a solution to problem \eqref{ex:2:fp}--\eqref{ex:2:fp:bc} satisfies
the fractional Euler--Lagrange equation
\begin{equation}
\label{el:ex2}
{_xD_1^\alpha}({^C_0D_x^\alpha}y(x)-wf(x))=0.
\end{equation}
Moreover, by Theorem~\ref{suff}, a solution to \eqref{el:ex2} is a
global minimizer to problem \eqref{ex:2:fp}--\eqref{ex:2:fp:bc}.
Therefore, solving equation \eqref{el:ex2} for $w\in[0,1]$, we are
able to obtain Pareto optimal solutions to problem
\eqref{ex:2:pr}--\eqref{ex:2:pr:bc}. In order to solve equation
\eqref{el:ex2}, firstly we use Corollary~2.1 of \cite{kilbas} to get the
following equation:
\begin{equation}
\label{1storder:FDELe}
{^C_0D_x^\alpha}y(x)-wf(x)=d(1-x)^{\alpha-1},\quad d\in \R.
\end{equation}
\begin{figure}
\centering
\includegraphics[scale=0.51]{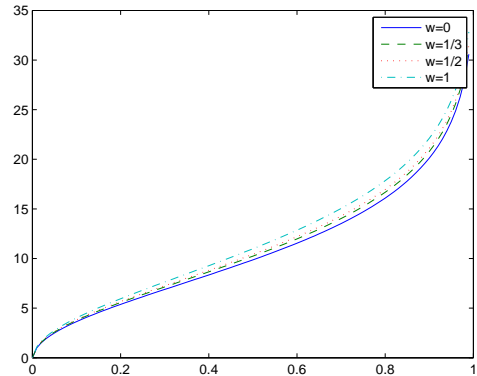}
\caption{\label{fig} The minimizer of Example~\ref{ex:2}
for $f(x) = \mathrm{e}^x$ and $\alpha = \frac{1}{2}$.}
\end{figure}
Equation~\eqref{1storder:FDELe} needs to be solved numerically. We did
numerical simulations using the MatLab solver \texttt{fode} for
linear Fractional-Order Differential Equations (FODE) with constant
coefficients, developed by Farshad Merrikh Bayat \cite{matlab}. The
results for $\alpha = 1/2$, $f(x) = \mathrm{e}^x$, and different values of
the parameter $w$ can be seen in Figure~\ref{fig}. Numerical results for
different values of $\alpha$ show that when $\alpha\rightarrow 1$
the fractional solution converges to the solution of the classical
problem of the calculus of variations.
\end{example}


\section*{Acknowledgments}

The authors would like to express their gratitude to Ivo Petras,
for having called their attention to \cite{matlab}
as well as for helpful discussions on numerical aspects and available
software packages for fractional differential equations.



\end{document}